\documentclass[12pt, article]{amsart}
\usepackage{graphicx} 
\usepackage{amsmath,amsthm,amssymb,latexsym,a4wide,tikz,multicol,tikz-cd, calc}
\usepackage{tikz-qtree,tikz-qtree-compat}
\usepackage{tikz}
\tikzset{font=\small}
\usepackage{enumerate}
\usetikzlibrary{matrix,arrows,automata,positioning,decorations.pathmorphing}
\usepackage{tikz-cd}
\usetikzlibrary{cd}
\usetikzlibrary{patterns}

\tikzset{main node/.style={circle,fill=white,draw,minimum size=0.1cm,inner sep=0pt},}
\newtheorem{theorem}{Theorem} [section]
\newtheorem{lemma}[theorem]{Lemma}

\newtheorem{proposition}[theorem]{Proposition}

\newtheorem{remark}[theorem]{Remark}
\theoremstyle{definition}

\newcommand{\Cay}{\operatorname{Cay}}
\DeclareMathOperator{\dom}{dom}
\DeclareMathOperator{\ran}{ran}


\usepackage[active]{srcltx}

\usepackage{enumerate}
\numberwithin{equation}{section}

\begin{document}

\title[A new approach to universal $F$-inverse monoids in enriched signature] { A new approach to universal $F$-inverse monoids in enriched signature}

%%%%
\author{Ganna Kudryavtseva}
\address{University of Ljubljana,
Faculty of Mathematics and Physics, Jadranska ulica 19, SI-1000 Ljubljana, Slovenia / Institute of Mathematics, Physics and Mechanics, Jadranska ulica 19, SI-1000 Ljubljana, Slovenia}
\email{ganna.kudryavtseva@fmf.uni-lj.si}
%%%

%%%%
\author{Ajda Lemut Furlani}
\address{Institute of Mathematics, Physics and Mechanics, Jadranska ulica 19, SI-1000 Ljubljana, Slovenia/ Faculty of Mathematics and Physics, Jadranska ulica 19, SI-1000 Ljubljana, Slovenia}
\email{ajda.lemut@imfm.si}
%%%

\sloppy

\thanks{The authors were supported by the Slovenian Research and Innovation Agency grant~P1-0288.}

\begin{abstract}  
We show that the universal $X$-generated $F$-inverse monoid $F(G)$, where $G$ is an $X$-generated group, introduced by Auinger, Szendrei and the first-named author, arises as a quotient inverse monoid of the Margolis-Meakin expansion $M(G, X\cup \overline{G})$ of $G$, with respect to the extended generating set $X\cup \overline{G}$, where $\overline{G}$ is a bijective copy of $G$ which encodes the $m$-operation in $F(G)$. The construction relies on a certain dual-closure operator on the semilattice of all finite and connected subgraphs containing the origin of the Cayley graph $\Cay(G, X\cup {\overline{G}})$ and leads to a new and simpler proof of the universal property of $F(G)$.
\end{abstract}

\maketitle

\section{Introduction}
An $F$-inverse monoid is an inverse monoid such that every $\sigma$-class, where $\sigma$ is the minimum group congruence, has a maximum element, with respect to the natural partial order. These monoids appear naturally and are useful in various mathematical contexts, see \cite{AKSz21} and references therein; for a solution of the finite $F$-inverse cover problem, see  \cite{ABO22}.

$F$-inverse monoids possess the additional unary operation $s\mapsto m(s)$ assigning to each element $s$ the maximum element $m(s)$ in its $\sigma$-class. It was observed by Kinyon \cite{K18} that $F$-inverse monoids in the enriched signature $(\cdot, \, ^{-1}, m, 1)$ form a variety of algebras. In \cite{AKSz21}, Auinger, Szendrei and the first-named author found a model for the $F$-inverse expansion $F(G)$ (in this paper denoted $F(G,X)$) of an $X$-generated group $G$, which is an upgrade of the Margolis-Meakin expansion $M(G,X)$ of $G$ \cite{MM89}  (for definitions and properties of  $M(G,X)$ and $F(G,X)$, see Subsection \ref{subs:revision}). A special case of this construction, with $G$ being the free $X$-generated group $FG(X)$, is a model of the free $X$-generated $F$-inverse monoid $FFI(X)$. The construction of $F(G,X)$, just as in the case with $M(G,X)$, involves certain subgraphs of the Cayley graph $\Cay(G,X)$ of the $X$-generated group $G$. Its key novel feature is that the requirement of the connectedness of subgraphs under consideration is dropped. The appropriate analogues of paths in $\Cay(G,X)$ are {\em journeys} where, along with traversing edges, it is allowed to jump between vertices, the jumps being captured by the $m$-operation of $F(G,X)$. The proof of the universal property of $F(G,X)$ from \cite{AKSz21} relies on assigning journeys in $\Cay(G,X)$ to terms of a suitable term algebra and evaluating them in $F(G,X)$. The proof is independent of the universal property of $M(G,X)$ and implies the latter, along with the universal property of the Birget-Rhodes expansion $B(G)$ \cite{BR89, Szendrei89} of $G$ (see \cite[Remark 4.8]{AKSz21}).

In this paper, we show that $F(G,X)$ arises as the canonical quotient inverse monoid  $M^{\wedge}(G, Y)$ of the Margolis-Meakin expansion  $M(G, Y)$ of $G$, with respect to the extended set of generators $Y=X\cup \overline{G}$, where $\overline{G}$ is a set in a bijection with $G$ and encodes the $m$-operation in $F(G, X)$. This quotient arises from a suitable $G$-invariant dual-closure operator $j\colon {\mathcal X}_Y\to {\mathcal X}_Y$ on the semilattice ${\mathcal X}_Y$ of all finite and connected subgraphs of $\Cay(G,Y)$, which contain the origin. Note that the underlying order of ${\mathcal X}_Y$ is the anti-inclusion order (see Proposition \ref{prop:universal}) and upon reversing this order our dual-closure operator can be equivalently looked at as a closure operator. We show that $M^{\wedge}(G, Y)$ is an $X$-generated $F$-inverse monoid and is canonically isomorphic to $F(G,X)$ (Propositions \ref{prop:aa} and \ref{prop:isom}), the latter being essentially due to the fact that the gaps in finite and not necessarily connected subgraphs containing the origin of $\Cay(G,X)$ are determined by the edges labeled by $\overline{G}$ in the corresponding  {\em closed} connected subgraphs of $\Cay(G,Y)$. Applying the universal property of $M(G,Y)$, we show in Theorem \ref{th:univ_prop1} that $M^{\wedge}(G, Y)$ has the same universal property as $F(G,X)$, which yields a new and simpler proof of the universal property of $F(G,X)$. Our arguments rely on the structure result for $E$-unitary inverse semigroups in terms of partial actions which is recalled in Subsection \ref{subs:e_unitary}.

When this work was nearly complete, we learned of the preprint version of the paper~\cite{Sz23} by Nora Szak\'{a}cs, which also treats quotients of the Margolis-Meakin expansions arising from closure operators, but with a different purpose. While we use dual-closure operators to give a new perspective on $F(G,X)$, \cite{Sz23} shows an equivalence of categories between certain closure operators and suitable $E$-unitary or $F$-inverse monoids. The work \cite{Sz23} separately considers closure operators on {\em not necessarily connected} subgraphs of $\Cay(G,X)$ (and also an appropriate analogue of Stephen's procedure in inverse monoids), in order to study presentations of $F$-inverse monoids in enriched signature. Our results, together with those of \cite{Sz23}, suggest that this setting can be alternatively handled using {\em connected} subgraphs of $\Cay(G,Y)$ (see Remark \ref{rem:F_inv}).

For the undefined notions in inverse semigroups we refer the reader to \cite{Lawson_book, Petrich}, and in universal algebra to \cite{BS81}.

\section{Preliminaries}\label{sec:prelim}
\subsection{$X$-generated algebraic structures}\label{subs:inv}
We say that a group (or an involutive monoid, or an inverse monoid, or an $F$-inverse monoid) is $X$-{\em generated} via the {\em assignment map} $\iota_S\colon X\to S$ if $S$ is generated by $\iota_S(X)$. A map $\varphi\colon S\to T$ between $X$-generated groups (or involutive monoids, or inverse monoids, or $F$-inverse monoids) is called {\em canonical}, if $\varphi\iota_S = \iota_T$. Let $(X\cup X^{-1})^*$ be the free involutive monoid on
$X\cup X^{-1}$ and $S$ an $X$-generated inverse monoid (in particular a group). For each $u\in (X\cup X^{-1})^*$ by $[u]_S$ (or simply $[u]$ when $S$ is understood) we denote the {\em value} of $u$ in $S$, that is, the image of $u$ under the canonical morphism $(X\cup X^{-1})^*\to S$; if $x\in X$, we have $[x]_S=\iota_S(x)$.

\subsection{Partial group actions and premorphisms}
\label{subs:premor}
Let $\leq$ denote the natural partial order on an inverse monoid.
A {\em premorphism} from a group $G$ to a inverse monoid $S$ is a map $\varphi:G \to S$, such that the following conditions hold:
\begin{itemize}
\item[(PM1)] $\varphi(1)=1$,
\item[(PM2)] $\varphi(g)\varphi(h) \leq \varphi(gh)$, for all $g,h\in G$,
\item[(PM3)] $\varphi(g^{-1}) = \varphi(g)^{-1}$, for all $g\in G$.
\end{itemize}

If $S$ is the symmetric inverse monoid ${\mathcal I}(X)$, we will denote $\varphi(g)(x)$ by $\varphi_g(x)$.
By a {\em partial map} $f\colon A\to B$ from a set $A$ to a set $B$ we mean a map $f\colon C\to B$ where $C\subseteq A$. For $a\in A$ we say that $f(a)$ is {\em defined} if $a\in C$.

Let $G$ be a group and $X$ a (non-empty) set. We say that $G$ {\em acts partially} on $X$ if there exists a partial map $G\times X \to X, \, \,(g,x)\mapsto g\cdot x$, which satisfies the following conditions:
\begin{itemize}
\item[(PA1)] $1\cdot x$ is defined and equals $x$, for all $x \in X$,
\item[(PA2)] if $g\cdot x$ and $g\cdot(h\cdot x)$ are defined, then $gh \cdot x$ is defined and $g\cdot(h\cdot x)=gh \cdot x$, for all $g,h \in G$ and $x\in X$,
\item[(PA3)] if $g\cdot x$ is defined, then $g^{-1}\cdot (g\cdot x)$ is defined and equals $x$, for all $g\in G$ and $x \in X$.
\end{itemize}

A partial action $G\times X\to X$, $(g,x)\mapsto g\cdot x$, gives rise to a premorhpism $\varphi\colon G \to {\mathcal I}(X)$  given by $\varphi_g(x) = g\cdot x$. The notions of a partial action of $G$ on $X$ and of a premorphism $G \to {\mathcal I}(X)$ are easily seen to be equivalent. For more background on partial group actions, we refer the reader to \cite{KL04}; for a comprehensive survey on partial actions to \cite{D19}.
  
\subsection{The structure of $E$-unitary inverse semigroups in terms of partial actions}\label{subs:e_unitary}
Let ${\mathcal P}$ be a poset. A non-empty subset $I$ of $\mathcal{P}$ is said to be an {\em order ideal}, if $x\leq y$ and $y\in I$ imply that  $x \in I$, for all $x,y\in {\mathcal{P}}$. A map $f\colon {\mathcal P}\to {\mathcal Q}$ between posets is called an {\em order isomorphism}, provided that $x\leq y$ if and only if $f(x)\leq f(y)$, for all $x,y\in {\mathcal P}$.

For a semilattice $Y$, by $\Sigma(Y)$ we denote the inverse monoid of all order-isomorphisms between order ideals of $Y$. Partial actions of a group $G$ on $Y$ by order isomorphisms between order ideals correspond to premoprhisms $G\to \Sigma(Y)$.

We now recall the variation of the McAlister structure result \cite{McA74} on $E$-unitary inverse semigroups in terms of partial actions \cite{PR79, KL04}. Suppose that a group $G$ acts partially on a semilattice $Y = (Y,\wedge)$ by order isomorphisms between order ideals and that $\varphi \colon G \to \Sigma(Y)$ is the associated premorphism. On the set
$$
Y \rtimes_\varphi G = \{(e,g) \in Y \times G \colon e \in \ran\varphi_g\}
$$
define the following operations:
$$
(e,g)(f,h)=(\varphi_g(\varphi_{g^{-1}}(e)\wedge f),gh), \,\,
(e,g)^{-1}= (\varphi_{g^{-1}}(e), g^{-1}).
$$
When $\varphi$ is understood, we suppress the index $\varphi$ and denote $Y \rtimes_\varphi G$ by $Y \rtimes G$.
Then $Y \rtimes G$ is an inverse semigroup with  $Y\simeq E(Y \rtimes G)$ via the map $y\mapsto (y,1)$. The natural partial order on it is given by $(e,g)\leq (f,h)$ if and only if $g=h$ and $e\leq f$. For $(e,g), (f,h)\in Y \rtimes G$ we have $(e,g) \mathrel{\sigma} (f,h)$ if and only if $g=h$, so that $Y \rtimes G$ is $E$-unitary and $(Y \rtimes G)/\sigma \simeq G$ via the map $(y,g)\mapsto g$. It is easy to see that $Y\rtimes G$ is a monoid if and only if $Y$ has a top element, $1_Y$, in which case the identity element of $Y\rtimes G$ is $(1_Y, 1)$. Furthermore, for each $E$-unitary inverse semigroup $S$ we have that $S\simeq E(S) \rtimes_{\varphi} S/\sigma$ via the map $s\mapsto (ss^{-1}, [s]_{\sigma})$ where $\varphi$ is the {\em underlying premorphism} of $S$, defined by setting, for all $g\in S/\sigma$, 
$$\dom\varphi_g = \{e \in E(S) \colon \text{there exists } s\in S \text{ with } [s]_{\sigma} = g \text{ such that } e \leq s^{-1}s\},$$
 $\varphi_g(e)=se s^{-1}$, where $e\in \dom\varphi_g$ and $s$ is such that  $[s]_{\sigma} = g$ and $e \leq s^{-1}s$, where $[s]_{\sigma}$ denotes the $\sigma$-class of $s$.

\subsection{Cayley graphs of groups}\label{subs:cayley}
The {\em Cayley graph} $\Cay(G,X)$ of an $X$-generated group $G$ is defined as the oriented graph ${\mathrm{V}}
\sqcup {\mathrm{E}}^+ \sqcup {\mathrm{E}}^- $, where ${\mathrm{V}}=G$ is the set of vertices, ${\mathrm{E}}^+= G \times X$ is the set of {\em positive edges} and ${\mathrm{E}}^-= G \times X^{-1}$ is the set of {\em negative edges}.
We set ${\mathrm{E}}= {\mathrm{E}}^+ \sqcup {\mathrm{E}}^- $.

For convenience, we will denote an edge $(g,x)$ by $(g,x,g[x])$. We let $\alpha(g,x,g[x])=g$, $\omega(g,x,g[x])=g[x]$ and $l(g,x,g[x])=x$ be the {\em beginning},  the {\em end} and the {\em label} of the edge $(g,x,g[x])$. There is the involution ${}^{-1}\colon  {\mathrm{E}}\to {\mathrm{E}}$, defined by $(g,x, g[x])^{-1}=(g[x],x^{-1},g)$. The edge $(g[x],x^{-1},g)$ should be thought of as `the same edge' as $(g,x, g[x])$ but `traversed in the opposite direction'.

A \emph{non-empty path} in $\Cay(G,X)$ is a sequence $e_1e_2\cdots e_n$ ($n\geq 1$) of
edges, for which $\omega(e_i)=\alpha(e_{i+1})$ for all $i\in \{1,\dots, n-1\}$. For the path $p=e_1\cdots e_n$ we set $\alpha(p)=\alpha(e_1)$ and $\omega(p) = \omega(e_n)$. The {\em inverse path} of $p$ is the path $p^{-1}=e_n^{-1}\cdots e_1^{-1}$. The {\em empty path} at a vertex $g$ is denoted by $\varepsilon_g$ and we set $\alpha(\varepsilon_g)=\omega(\varepsilon_g) = g$. Two paths, $p$ in $q$, in $\Cay(G,X)$ are called {\em coterminal} if $\alpha(p)=\alpha(q)$ and $\omega(p)=\omega(q)$.
The {\em label} of the path $p=e_1\cdots e_n$, where $n\geq 1$, is defined by $l(p)=l(e_1)\cdots l(e_n)\in (X\cup X^{-1})^+$ while $l(\varepsilon_g)=1$ for all $g\in G$. 
A {\em subgraph} $\Gamma$ of $\Cay(G,X)$ is a subset $\Gamma \subseteq \Cay(G,X)$, which is closed under $\alpha$, $\omega$ and $^{-1}$. By ${\mathrm{V}}(\Gamma)$ and ${\mathrm{E}}(\Gamma)$ we denote the sets of vertices and edges of the subgraph $\Gamma$. Any subset $P \subseteq \Cay(G,X)$ yields a unique subgraph of $\Cay(G,X)$ called the {\em subgraph spanned by $P$} and denoted by $\langle P \rangle$. If $p$ is a path in $\Cay(G,X)$, then the graph $\langle p \rangle$ spanned by $p$ is defined as the graph spanned by the edges of $p$.
Note that if $P$ is finite, so is $\langle P \rangle$. A subgraph $\Gamma$ is {\em connected}, if for any two vertices $u,v\in \Gamma$ there exists a path $p$ in $\Gamma$ which begins in $u$ and ends in $v$.

\subsection{The universal inverse monoid $M(G,X)$ and $F$-inverse monoid $F(G,X)$ of an $X$-generated group $G$}\label{subs:revision}
Let $G$ be an $X$-generated group and $\Cay(G,X)$ its Cayley graph. We introduce the following notation:
\begin{itemize}
\item ${\mathcal X}_{X}$ -- the set of all finite connected subgraphs of $\Cay(G,X)$ which contain the origin,
\item ${\tilde{\mathcal X}}_{X}$ -- the set of all finite (and not necessarily connected) subgraphs of $\Cay(G,X)$ which contain the origin.
\end{itemize}
These are semilattices with $A \leq B$ if and only if $A\supseteq B$; their top element is the graph $\Gamma_1$ with only one vertex, $1$, and no edges.
We put:
\begin{align*}
M(G,X) & = \{(\Gamma,g)\colon \Gamma \in {\mathcal X}_{X} \text{ and } g\in {\mathrm{V}}(\Gamma)\},\\
F(G,X) & = \{(\Gamma,g)\colon \Gamma \in {\tilde{\mathcal X}}_{X} \text{ and } g\in {\mathrm{V}}(\Gamma)\}
\end{align*}
and define the operations on $M(G,X)$ and $F(G,X)$ by
$$
(A,g)(B,h) = (A\cup gB, gh), \,\,\, (A,g)^{-1} = (g^{-1}A, g^{-1}).
$$
Then $M(G,X)$ is an $E$-unitary inverse monoid called the {\em Margolis-Meakin expansion} of the $X$-generated group $G$ \cite{MM89}, and $F(G,X)$ is the $F$-inverse monoid $F(G)$ introduced by Auinger, Szendrei and the first-named author in \cite{AKSz21}. 
In the following proposition we collect some properties of $M(G,X)$ and $F(G,X)$.
\begin{proposition} \label{prop:universal}\cite{MM89, AKSz21} Let $G$ be an $X$-generated group.
\begin{enumerate}
\item \label{i:p1} The identity element of each of $M(G,X)$ and $F(G,X)$ is $(\Gamma_1,1)$.
\item \label{i:p2}  $M(G,X)$ is an $X$-generated inverse monoid and $F(G,X)$ is an $X$-generated $F$-inverse monoid 
via the assignment map $x\mapsto (\Gamma_x, [x])$, where $\Gamma_x$ is the graph with two vertices, $1$ and $[x]$, and the positive edge $(1,x,[x])$.
\item \label{i:p3} The natural partial order on each of $M(G,X)$ and $F(G,X)$ is given by $(A,g)\leq (B,h)$ if and only if $g=h$ and $B\subseteq A$.
\item \label{i:p4} Let $(A,g), (B,h)$ be elements of one of $M(G,X)$ or $F(G,X)$. Then $(A,g) \mathrel{\sigma} (B,h)$ if and only if $g=h$, which implies that $M(G,X)/\sigma \simeq G$ and  $F(G,X)/\sigma \simeq G$ via the canonical morphism $(\Gamma, g)\mapsto g$.
\item \label{i:p5} In $F(G,X)$ the maximum element of the ${\mathcal{\sigma}}$-class of $(A,g)$ is $(\{1,g\},g)$ where $\{1,g\}$ is the graph with vertices $1,g$ and no edges.
\item \label{i:p6} \begin{itemize}
\item   $M(G,X)= {\mathcal X}_X \rtimes G$, where the underlying premorphisms $\varphi\colon G\to \Sigma({\mathcal X}_{X})$ is given, for each $g\in G$, by $\dom\varphi_g = \{\Gamma \in {\mathcal X}_{X}\colon  g^{-1}\in {\mathrm{V}}(\Gamma)\}$ and $\varphi_g(\Gamma) = g\Gamma$ for all $\Gamma\in \dom\varphi_g$.
\item $F(G,X)  = {\tilde{\mathcal X}}_{X} \rtimes G$, where the underlying premorphism $\tilde\varphi\colon G\to \Sigma(\tilde{\mathcal X}_{X})$ is given similarly as above for $M(G,X)$, with ${\mathcal X}_{X}$ replaced by ${\tilde{\mathcal X}}_{X}$.
\end{itemize}
\item (Universal properties of $M(G,X)$ and $F(G,X)$) \label{i:p7} Let $S$ be an $X$-generated $E$-unitary inverse monoid
(respectively, an $X$-generated $F$-inverse monoid) such that there is a canonical morphism
$\nu\colon G\to S/\sigma$. Then there is a canonical morphism $\varphi\colon M(G,X) \to S$ (respectively, $F(G,X)\to S$) such that the following diagram of canonical morphisms of $X$-generated inverse monoids (respectively, of $X$-generated $F$-inverse monoids) commutes:
\begin{center}
 \begin{tikzcd}
 U(G,X) \arrow[r, "\varphi"] \arrow[d] & S \arrow[d]\\
 G\arrow[r, "\nu"] & S/\sigma
 \end{tikzcd}
\end{center}
where $U(G,X)$ is one of $M(G,X)$ or $F(G,X)$.
\item \label{i:p8} Let $p$ be a path from $1$ to $g$ in $\Cay(G,X)$. Then $[l(p)]_{M(G,X)}=(\langle p\rangle,g)$.
\item \label{i:p9} Let $p$ and $q$ be paths in $\Cay(G,X)$. Then  $[l(p)]_{M(G,X)} \mathrel{\sigma} [l(q)]_{M(G,X)}$ if and only if $[l(p)]_G = [l(q)]_G$. 
\end{enumerate}
\end{proposition}

We remark that the results in \eqref{i:p6} above are not explicitly stated in \cite{MM89, AKSz21}, but easily follow from the results therein and are known.

\section{Dual-closure operators on semilattices and quotients of partial action products}\label{sec:closure}
Recall that a {\em dual-closure operator}  (or an {\em interior operator}) on a poset \mbox{$({\mathcal P},\leq)$} is a map $j\colon {\mathcal P}\to {\mathcal P}$, which satisfies the following conditions, for all $x,y\in {\mathcal P}$:
\begin{enumerate}
\item[(Cl1)] $j(x)\leq x$,
\item[(Cl2)] if $x\leq y$ then $j(x)\leq j(y)$,
\item[(Cl3)] $j(j(x)) = j(x)$.
\end{enumerate}

From now on let $X$ be a semilattice and $j\colon X\to X$ a dual-closure operator on it.

\begin{lemma}\label{lem:closure1} 
For all $x,y\in X$ we have $j(j(x)\wedge j(y)) = j(x\wedge y)$.
\end{lemma}

\begin{proof} 
Since $j(x)\leq x$ and $j(y)\leq y$ by (Cl1), we have $j(x)\wedge j(y)\leq x\wedge y$. From (Cl2), it follows that $j(j(x)\wedge j(y))\leq j(x\wedge y)$.  For the opposite inequality observe that, since $x\wedge y\leq x$, we have $j(x\wedge y)\leq j(x)$, by (Cl2), and similarly $j(x\wedge y)\leq j(y)$. Hence $j(x\wedge y)\leq j(x)\wedge j(y)$. Applying (Cl2), we write $j(j(x\wedge y))\leq j(j(x)\wedge j(y))$ which, in view of (Cl3), yields $j(x\wedge y)\leq j(j(x)\wedge j(y))$, as required.
\end{proof}

We define the equivalence relation ${\rho_{j}}$ on $X$ by
$$
x \mathrel{\rho_{j}} y \iff j(x)=j(y).
$$
It is easy to see that $\mathrel{\rho_{j}}$ is a congruence on $X$.  It will be convenient to identify the quotient semilattice $X/\rho_{j}$  with the semilattice $(j(X), \bar\wedge)$ with the operation $j(x) \, \bar\wedge \, j(y) = j(x\wedge y)$.

Let $G$ be a group acting partially on the semilattice $X$ by order isomorphisms between order ideals.
We say that a dual-closure operator $j$ on $X$ is {\em $G$-invariant} provided that if $g\cdot x$ is defined then $j(g\cdot x) = g \cdot j(x)$, for all $g\in G$ and $x\in X$.  If $j$ is a $G$-invariant dual-closure operator on $X$, define the relation $\tilde{\rho}_{j}$ on the inverse semigroup $X \rtimes G$ by 
\begin{equation}\label{eq:rho_tilde}
(e,g) \mathrel{\tilde{\rho}_{j}} (f,h) \iff g=h \text{ and } j(e)=j(f).
\end{equation}

\begin{proposition}\label{prop:congruence} 
The relation $\tilde{\rho}_{j}$ is a congruence on $X \rtimes G$, which is contained in $\sigma$. Moreover $(X \rtimes G)/{\tilde{\rho}_{j}} \simeq j(X) \rtimes G$.
\end{proposition}

\begin{proof}
Obviously, $\tilde{\rho}_{j}$ is an equivalence relation, contained in $\sigma$. Suppose that $(e,g) \mathrel{\tilde{\rho}_{j}} (f,h)$. Then
$g=h$ and $j(e) = j(f)$. Let $(d,s) \in X \rtimes G$ and show that $(e,g)(d,s) \mathrel{\tilde{\rho}_{j}} (f,h)(d,s)$ or, equivalently, $(\varphi_g(\varphi_{g^{-1}}(e)\wedge d),gs) \mathrel{\tilde{\rho}_{j}}(\varphi_g(\varphi_{g^{-1}}(f)\wedge d),gs)$. It suffices to show that 
$j(\varphi_g(\varphi_{g^{-1}}(e)\wedge d)) = j(\varphi_g(\varphi_{g^{-1}}(f)\wedge d))$. The left-hand side rewrites to
$$
\varphi_g(j(\varphi_{g^{-1}}(e)\wedge d)) = \varphi_g ( j(\varphi_{g^{-1}}(e)) \bar\wedge j(d)) =  \varphi_g(j(\varphi_{g^{-1}}(j(e))\wedge j(d)))$$ 
and, similarly, the right-hand side to $\varphi_g(j( \varphi_{g^{-1}}(j(f))\wedge j(d)))$. Since $j(e)=j(f)$, the two expressions coincide. Likewise, one shows that  $(d,s)(e,g) \mathrel{\tilde{\rho_j}} (d,s)(f,h)$.
The map $(X \rtimes G)/{\tilde{\rho}_{j}}\to j(X) \rtimes G$ given by $[(e,g)]_{\tilde{\rho}_{j}} \mapsto (j(e), g)$, is obviously well defined, and it is routine to check that it is an isomorphism of semigroups.
\end{proof}

\section{A new approach to the universal $F$-inverse monoid $F(G,X)$}\label{sec:new_approach}
\subsection{The inverse monoid $M(G,Y)$.} \label{subs:m}
Let $X$ be a nonempty set and $G$ an $X$-generated group. In what follows, we will need to consider $G$ also with respect to another generating set, so to distinguish between the assignment maps for different generating sets, we will denote the assignment map $X\to G$ by $\iota_{G,X}$. Recall that we abbreviate $\iota_{G,X}(x)$ by $[x]$. Let, further, $\overline{G}$ be a disjoint copy of $G$, and we fix the bijection $g\mapsto \overline{g}$ between $G$ and $\overline{G}$.

We will consider the group $G$ also with respect to the `extended' generating set $Y=X\cup \overline{G}$ via the asisgnment map $\iota_{G,Y} \colon Y \to G$  given by 
\begin{equation}\label{eq:fgx}
\begin{aligned}
x & \mapsto [x],  && \text{ if }  x\in X,  \\
\overline{g} & \mapsto  g, && \text{ if }  \overline{g}\in \overline{G}.
\end{aligned}
\end{equation}
 
In particular, for all $x\in X$ we have $\iota_{G,Y}(x) = \iota_{G,X}(x) = [x]$.

For $\overline{g}\in \overline{G}$ we define $ \Gamma_{\overline{g}}$ to be the graph with two vertices, $1$ and $g$, and one positive edge $(1,\overline{g},g)$.
Then the inverse monoid $M(G,Y)= {\mathcal X}_{Y} \rtimes G$  is $Y$-generated via the  map $\iota_{M(G,Y)} \colon Y\to M(G,Y)$ given by
\begin{equation}\label{eq:def_gamma}
\begin{aligned}
x \mapsto (\Gamma_x, [x]),  && \text{ if }  x\in X,  \\
\overline{g} \mapsto (\Gamma_{\overline{g}},g), && \text{ if }  \overline{g}\in \overline{G},
\end{aligned}
\end{equation}
and its identity element is $(\Gamma_1, 1)$.

\subsection{The $F$-inverse monoid $M^{\wedge}(G,Y)$}\label{subs:m_wedge}
We call a subgraph $\Gamma\in {\mathcal X}_Y$ {\em closed} provided that it satisfies the following condition:
\begin{enumerate}
\item[(C)] If $a,b\in G$ are such that $a, b\in {\mathrm{V}}(\Gamma)$, then  $(a,\overline{a^{-1}b}, b)\in  {\mathrm{E}}(\Gamma)$.
\end{enumerate}

For $\Gamma \in {\mathcal X}_Y$ we put $\Gamma^{\wedge}$ to be smallest closed graph in ${\mathcal X}_Y$ which contains $\Gamma$. It is clearly well defined and we have ${\mathrm{V}}(\Gamma^{\wedge}) = {\mathrm{V}}(\Gamma)$ and ${\mathrm{E}}^+(\Gamma^{\wedge}) = {\mathrm{E}}^+(\Gamma) \cup \{(a,\overline{a^{-1}b}, b)\colon a,b \in {\mathrm{V}}(\Gamma)\}$. 
It is easy to see that $j\colon {\mathcal X}_Y\to {\mathcal X}_Y$, $\Gamma\mapsto \Gamma^{\wedge}$, is a $G$-invariant dual-closure operator.
We put ${\mathcal X}_Y^{\wedge} = j({\mathcal X}_Y)$.

The congruence $\tilde{\rho_j}$ of \eqref{eq:rho_tilde} on $M(G,Y)={\mathcal X}_{Y} \rtimes G$ is given by
\begin{equation}\label{eq:tilde_rho_j}
(A, g) \mathrel{\tilde\rho_j} (B,h) \iff g=h \text{ and } A^{\wedge} = B^{\wedge}.
\end{equation}

By Proposition \ref{prop:congruence} the quotient inverse monoid $M(G,Y)/\tilde{\rho_j}$  is isomorphic to ${\mathcal X}_{Y}^{\wedge} \rtimes G$, which we  denote by $M^{\wedge}(G,Y)$. This is a $Y$-generated  inverse monoid via the assignment map 
\begin{equation}\label{eq:def_gamma1}
\begin{aligned}
x \mapsto (\Gamma_x^{\wedge}, [x]),  && \text{ if }  x\in X,  \\
\overline{g} \mapsto (\Gamma_{\overline{g}}^{\wedge},g), && \text{ if }  \overline{g}\in \overline{G},
\end{aligned}
\end{equation}
and its identity element is $ (\Gamma^{\wedge}_{1},1) =(\Gamma_{\overline{1}},1)$. The operations on it are given by
\begin{equation}\label{eq:prod_sg}
(A,g)(B,h) = ((A\cup gB)^{\wedge}, gh),    
\end{equation}
\begin{equation}\label{eq:inv_sg}
(A,g)^{-1} = (g^{-1}A, g^{-1}).    
\end{equation}

\begin{remark}
{\em It is not hard to show that the congruence $\tilde{\rho}_j$ is generated by the relations $\overline{[x]}\geq x$, $\overline{gh} \geq \overline{g}\,\overline{h}$ and $\overline{g^{-1}} = \overline{g}^{-1}$, where $x\in X$, $g,h\in G$, but we will not use this fact in our arguments.}
\end{remark}

\begin{proposition} \label{prop:aa} $M^{\wedge}(G,Y)$ is an $X$-generated (in the enriched signature $(\cdot, ^{-1}, m, 1)$) $F$-inverse monoid with 
$m(\Gamma, g) = (\Gamma_{\overline{g}}^{\wedge},g)$, for all $(\Gamma, g)\in M^{\wedge}(G,Y)$.
\end{proposition}

\begin{proof} 
The second claim is clear by the description of $\sigma$ and the fact that $\Gamma_{\overline{g}}^{\wedge}$ is minimum among all the graphs in ${\mathcal X}^{\wedge}_Y$, which have $1$ and $g$ as vertices. For the first claim, it suffices to show that each $(\Gamma_{\overline{g}}^{\wedge}, g)$, where $g$ runs through $G\setminus \{1\}$, can be written via the generators $(\Gamma^{\wedge}_x, [x])$ in the enriched signature. We write $g=[x_1^{\varepsilon_1}\cdots x_n^{\varepsilon_n}]$, where $n\geq 1$, all $x_i\in X$ and $\varepsilon_i\in \{1,-1\}$, then $(\Gamma_{\overline{g}}^{\wedge}, g) = m((\Gamma_{x_1}^{\wedge},[x_1])^{\varepsilon_1}\cdots (\Gamma_{x_n}^{\wedge},[x_n])^{\varepsilon_n})$.
\end{proof}

\begin{proposition}\label{prop:isom} The $X$-generated $F$-inverse monoids $M^{\wedge}(G,Y)$ and $F(G,X)$ are canonically isomorphic.
\end{proposition}

\begin{proof}
Define the map ${\bf f}\colon {\mathcal X}^{\wedge}_{Y} \to {\tilde{\mathcal X}}_{X}$ by ${\bf f}(\Gamma) = \Gamma \cap \Cay(G,X)$ where $\Gamma \cap \Cay(G,X)$ is the graph obtained from $\Gamma$ by erasing all its edges labeled by $\overline{G}$ and their inverse edges labeled by ${\overline{G}}^{-1}$. This map is injective as $\Gamma$ can be reconstructed from $\Gamma \cap \Cay(G,X)$ by adding to the latter all the edges of $\Cay(G,Y)$ labeled by $\overline{G}$ and ${\overline{G}}^{-1}$ between its vertices. It is clearly surjective. Since, in addition, 
$$(\Gamma_1 \cap \Cay(G,X)) \cup (\Gamma_2 \cap \Cay(G,X)) = (\Gamma_1\cup \Gamma_2)^{\wedge} \cap \Cay(G,X),$$ 
for all $\Gamma_1, \Gamma_2 \in  {\mathcal X}_{Y}^\wedge$,  it is an isomorphism of semilattices. It is immediate that ${\bf f}$ respects the partial action of $G$, that is, for all $g\in G$ and $\Gamma \in {\mathcal X}_{Y}^{ \wedge}$ we have that $\varphi_g(\Gamma)$ is defined if and only if so is $\tilde{\varphi}_g({\bf f}(\Gamma))$, in which case we have ${\bf f}(\varphi_g(\Gamma)) = \tilde{\varphi}_g({\bf f}(\Gamma))$. Here $\varphi\colon G\to \Sigma({\mathcal X}_{Y}^{\wedge})$ and $\tilde{\varphi} \colon G\to \Sigma({\tilde{\mathcal X}}_{X})$ are the underlying premorphisms of $M^{\wedge}(G,Y)$ and $F(G,X)$ (see part \eqref{i:p6} of Proposition \ref{prop:universal}). It now easily follows that the map
$M^{\wedge}(G,Y) = {\mathcal X}^{\wedge}_{Y} \rtimes G \to 
{\tilde{\mathcal X}}_{X} \rtimes G = F(G,X)$, given by $(\Gamma,g) \mapsto ({\bf f}( \Gamma), g)$, is an isomorphism of $F$-inverse monoids. That it is canonical is immediate by the construction.
\end{proof}

We now prove the universal property of $M^{\wedge}(G,Y)$.
\begin{theorem}\label{th:univ_prop1} For any $X$-generated group $G$ and any $X$-generated $F$-inverse monoid $F$ (in the signature $(\cdot, ^{-1}, m, 1)$) such that there is a canonical morphism
$\nu\colon G\to F/\sigma$, there is a canonical morphism $\varphi\colon M^{\wedge}(G,Y) \to F$ such that the diagram of canonical morphisms of $X$-generated $F$-inverse monoids
\begin{center}
 \begin{tikzcd}
 M^{\wedge}(G,Y) \arrow[r, "\varphi"] \arrow[d] & F \arrow[d]\\
 G\arrow[r, "\nu"] & F/\sigma
 \end{tikzcd}
\end{center}
commutes. 
\end{theorem}

\begin{proof}
Since in this proof, for some algebras $A$ (which are groups,  inverse monoids or $F$-inverse monoids) we work with two generating sets, $X$ and $Y$, we denote the corresponding assignment maps by $\iota_{A,X}$ and $\iota_{A,Y}$, respectively. Because $F$ is an $X$-generated $F$-inverse monoid, it is an $(X \cup \{m(s)\colon s\in F\})$-generated inverse monoid (this is easy to show and known, see \cite[Section 3]{AKSz21}). Let $\tau_F\colon F/\sigma \to F$ be the map which assigns to each $f\in F/\sigma$ the maximum element $\tau_F(f)$ of the $\sigma$-class of $F$ which projects onto $f$. Then $F$ is a $Y$-generated inverse monoid via the assignment map $\iota_{F,Y}\colon Y\to F$, such that $\iota_{F,Y}$ coincides with $\iota_{F,X}$ on $X$ and $\iota_{F,Y}(\overline{g}) = \tau_F\nu(g)$ for $\overline{g}\in \overline{G}$. By the universal property of $M(G,Y)$ (see part \eqref{i:p7} of Proposition \ref{prop:universal}), there is a canonical morphism of $Y$-generated inverse monoids $\psi\colon M(G,Y) \to F$, such that the following diagram of canonical morphisms of $Y$-generated inverse monoids commutes:
\begin{center}
\begin{tikzcd}
M(G,Y) \arrow[r, "\psi"] \arrow[d] & F \arrow[d]\\
G\arrow[r, "\nu"] & F/\sigma
\end{tikzcd}
\end{center}

We show that $\psi\colon M(G,Y) \to F$ factors through the canonical quotient map $\pi\colon M(G,Y)\to M^{\wedge}(G,Y)$, as is illustrated below: 
\begin{center}
$$\begin{tikzcd}
{M(G,Y)} \\
& {M^{\wedge}(G,Y)} & F \\
& G & {F/\sigma}
\arrow["\pi", from=1-1, to=2-2]
\arrow["\varphi", from=2-2, to=2-3]
\arrow[bend right=16, from=1-1, to=3-2]
\arrow[from=2-3, to=3-3]
\arrow["\nu", from=3-2, to=3-3]
\arrow["\psi", bend left=14, from=1-1, to=2-3]
\arrow[from=2-2, to=3-2]
\end{tikzcd}
$$
\end{center}
In view of \eqref{eq:tilde_rho_j} and since $M^{\wedge}(G,Y) \simeq M(G,Y)/\tilde{\rho}_j$, it suffices to show that $\psi(A,g) = \psi(A^{\wedge},g)$ for all $(A,g)\in M(G,Y)$. Since the graph $A^{\wedge}$ is obtained from the graph $A$ by adding to it finitely many edges, there is a finite sequence $A=A_0, A_1,\dots, A_n=A^{\wedge}$ of graphs in ${\mathcal X}_{Y}$ such that, for each $i=0,\dots, n-1$, the graph $A_{i+1}$ is obtained from the graph $A_i$ by adding to it a single positive edge $(a,\overline{g}, ag)$ (and also its inverse negative edge). 
It thus suffices to prove that $\psi(B,g) = \psi(C,g)$ where the graph $C$ is obtained from the graph $B$ by adding to it a single positive edge $e=(a,\overline{g}, ag)$ (and also its inverse negative edge) between $a,ag\in {\mathrm{V}}(B)$. Since $B$ is connected, there is a path, $p$, in $B$ with $\alpha(p)=\alpha(e)$ and $\omega(p) = \omega(e)$. Let $w'ew''$ be a spanning path in $C$ from the origin to $g$. 
Then the path $w=w'pw''$ spans $B$ and the path
$\tilde{w}=w'ep^{-1}pw''$ spans $C$,  moreover, $w$ and $\tilde{w}$ are coterminal from the origin to $g$. Let $s,t, u\in (Y\cup Y^{-1})^*$ be the labels of $w'$, $w''$ and $p$, respectively. Then  $l = sut$ and $\tilde{l} = s\overline{g}u^{-1}ut$ are the labels of $w$ and $\tilde{w}$, respectively. 

Since $(B,g)$ (respectively, $(C,g)$) equals the value in $M(G,Y)$ of the label of any path in $\Cay(G,Y)$ which spans $B$ (respectively, $C$) from the origin to $g$ (by part \eqref{i:p8} of Proposition \ref{prop:universal}), we have that $\psi(B,g) = [l]_F$ and $\psi(C,g) = [\tilde{l}]_F$ (the evaluations are taken in the $Y$-generated inverse monoid $F$). We then have
$\psi(B,g)= [s]_F[u]_F[t]_F$ and $\psi(C,g) = [s]_F[\overline{g}]_F[u]_F^{-1}[u]_F[t]_F$.
But $[u]_{M(G,Y)} \mathrel{\sigma} [\overline{g}]_{M(G,Y)}$ as $p$ and $e$ are coterminal (by part \eqref{i:p9} of Proposition \ref{prop:universal}), moreover, $[u]_{F} \leq [\overline{g}]_{F}$ as $[\overline{g}]_{F} = \tau_F\nu(g)$ is the maximum element in its $\sigma$-class. It follows that $[\overline{g}]_F[u]_F^{-1}[u]_F = [u]_F$, which implies the desired equality $\psi(B,g)= \psi(C,g)$.
Therefore, there is a well defined canonical morphism of $Y$-generated inverse monoids $\varphi\colon M^{\wedge}(G,Y)\to F$ such that 
$\varphi\pi = \psi$. Since  $[\overline{g}]_F = \varphi( \Gamma_{\overline{g}}^{\wedge},g)$, for all $g\in G$, it follows that $\varphi$ preserves the $m$-operation, and is thus a canonical morphism of $X$-generated $F$-inverse monoids.
\end{proof}

Theorem \ref{th:univ_prop1} and Propositions \ref{prop:aa} and \ref{prop:isom} provide a new proof of the universal property of $F(G,X)$. 

\begin{remark}\label{rem:F_inv}
{\em Combining our results with those of \cite{Sz23}, one can show that any $X$-generated $F$-inverse monoid $F$ (in the enriched signature $(\cdot, \, ^{-1}, m, 1)$), looked at as a $Y$-generated inverse monoid (as in the proof of Theorem \ref{th:univ_prop1}), arises as a canonical quotient of $M(G,Y)$, where $G=F/\sigma$, and generators from $\overline{G}$ are mapped onto respective maximal elements of $\sigma$-classes of $F$. This suggests that presentations of $F$-inverse monoids in enriched signature can be studied by the usual tools (Stephen's procedure \cite{S90}) developed for inverse monoids.} 
\end{remark}

\section*{Acknowlegements} We thank the referee for useful comments.

\end{document}